\documentclass[11pt,a4paper]{article}
\usepackage{amsmath}
\usepackage{amsfonts}

\newtheorem{theorem}{Theorem}

\newtheorem{corollary}[theorem]{Corollary}

\newtheorem{proposition}[theorem]{Proposition}

\newenvironment{proof}[1][Proof]{\noindent\textbf{#1.} }{\ \rule{0.5em}{0.5em}}
\input{tcilatex}
\begin{document}

\title{A convergence theorem for harmonic measures with applications to
Taylor series}
\date{}
\author{Stephen J. Gardiner and Myrto Manolaki}
\maketitle

\begin{abstract}
Let $f$ be a holomorphic function on the unit disc, and $(S_{n_{k}})$ be a
subsequence of its Taylor polynomials about $0$. It is shown that the
nontangential limit of $f$ and lim$_{k\rightarrow \infty }S_{n_{k}}$ agree
at almost all points of the unit circle where they simultaneously exist.
This result yields new information about the\ boundary behaviour of
universal Taylor series. The key to its proof lies in a convergence theorem
for harmonic measures that is of independent interest.
\end{abstract}

\section{Introduction}

\footnotetext{%
\noindent 2010 \textit{Mathematics Subject Classification } 30B30, 30C85,
30K05, 31A15, 31B20.}Let $f$ be a holomorphic function on the unit disc $%
\mathbb{D}$. We assume that its Taylor series about $0$ has radius of
convergence $1$ and denote by $S_{n}$ the partial sum of this series up to
degree $n$. It is natural to ask how the boundary behaviour of $f$ at a
subset $A$ of the unit circle $\mathbb{T}$ constrains the functions on $A$
that can arise as $\lim_{k\rightarrow \infty }S_{n_{k}}$ for some
subsequence $(S_{n_{k}})$ of $(S_{n})$.

It turns out that even in the simplest situation, where $f$ is holomorphic
on $\mathbb{C}\backslash \{1\}$, the sequence $(S_{n})$ typically has
chaotic behaviour on Dirichlet subsets of $\mathbb{T}$, that is, compact
sets on which $(z^{n_{k}})$ converges uniformly to $1$ for some subsequence $%
(n_{k})$ of the natural numbers. More precisely, Beise, Meyrath and M\"{u}%
ller \cite{BMM} have shown recently that, given any Dirichlet set $A\subset 
\mathbb{T}\backslash \{1\}$, there is a residual subset of the space of
holomorphic functions on $\mathbb{C}\backslash \{1\}$ (endowed with the
topology of local uniform convergence), each member $f$ of which has the
properties that:

\begin{enumerate}
\item[(i)] for each continuous function $h$ on $A$ there is a subsequence $%
(S_{n_{k}})$ that converges uniformly to $h$ on $A$;

\item[(ii)] there is a subsequence $(S_{m_{k}})$ that converges locally
uniformly to $f$ on $\mathbb{T}\backslash \{1\}$.
\end{enumerate}

Dirichlet sets $A$ can have Hausdorff dimension $1$ but cannot have positive
arc length measure $\sigma (A)$ (see, for example, p.171 of \cite{Kah}).
This leaves open the question of whether property (i) above can occur on
subsets $A\subset \mathbb{T}$\ of positive measure. We show below that this
cannot happen, even where the boundary values of $f$ exist merely as
nontangential limits. Let $\mathrm{nt}\lim_{z\rightarrow \zeta }f(z)$ denote
the nontangential limit of $f$ at a point $\zeta \in \mathbb{T}$, wherever
it exists (finitely).

\begin{theorem}
\label{main}Given a holomorphic function $f$ on $\mathbb{D}$ and a
subsequence $(S_{n_{k}})$ of the partial sums of its Taylor series about $0$%
, let%
\begin{equation*}
E=\{\zeta \in \mathbb{T}:S(\zeta ):=\lim_{k\rightarrow \infty
}S_{n_{k}}(\zeta )\text{ exists}\}
\end{equation*}%
and%
\begin{equation*}
F=\{\zeta \in \mathbb{T}:f(\zeta ):=\mathrm{nt}\lim_{z\rightarrow \zeta }f(z)%
\text{ exists}\}.
\end{equation*}%
Then $S=f$ almost everywhere $(\sigma )$ on $E\cap F$.
\end{theorem}

A classical result in this area is Abel's Limit Theorem, which says that, if 
$(S_{n}(\zeta ))$ converges for some $\zeta \in \mathbb{T}$, then $\mathrm{nt%
}\lim_{z\rightarrow \zeta }f(z)$ exists, and the two limits agree. If we
merely know that a subsequence $(S_{n_{k}}(\zeta ))$ converges, no
conclusion about the boundary behaviour of $f$ at $\zeta $ may be drawn.
Indeed, for a typical holomorphic function $f$ on $\mathbb{D}$, any
continuous function on $\mathbb{T}$ is the pointwise limit of a suitable
subsequence $(S_{n_{k}})$. (See the properties of the collection $\mathcal{U}%
_{0}(\mathbb{D},0)$ noted below.) Nevertheless, Theorem \ref{main} still
shows that $\lim_{k\rightarrow \infty }S_{n_{k}}(\zeta )$ and $\mathrm{nt}%
\lim_{z\rightarrow \zeta }f(z)$ must agree almost everywhere on the set
where they simultaneously exist.

Theorem \ref{main} fails if we replace nontangential limits by radial
limits. To see this, let $F$ be a closed nowhere dense subset of $\mathbb{T}$%
\ such that $\sigma (F)>0$. Then, by Theorem 1.2 of Costakis \cite{Co},
there is a holomorphic function $f$ on $\mathbb{D}$ which has radial limit $%
0 $ at each point of $F$ and such that some subsequence $(S_{n_{k}})$\
converges pointwise to $1$ on $\mathbb{T}$.

Now let $f$ be a holomorphic function on a proper subdomain $\omega $ of $%
\mathbb{C}$, let $\xi \in \omega $, $r_{0}=\mathrm{dist}(\xi ,\mathbb{C}%
\backslash \omega )$ and $D_{0}$ denote the open disc $D(\xi ,r_{0})$ of
centre $\xi $ and radius $r_{0}$. Further, let $S_{n}(f,\xi )$ denote the
partial sum up to degree $n$ of the Taylor series of $f$ about $\xi $.
Following Nestoridis \cite{Ne} we call this series \textit{universal}, and
write $f\in \mathcal{U}(\omega ,\xi )$, if for every compact set $K\subset 
\mathbb{C}\backslash \omega $ that has connected complement, and every
continuous function $h$ on $K$ that is holomorphic on $K^{\circ }$, there is
a subsequence $(S_{n_{k}}(f,\xi ))$ that converges uniformly to $h$ on $K$.
Similarly, we write $f\in \mathcal{U}_{0}(\omega ,\xi )$ if $f$ satisfies
the corresponding condition in which we only consider compact sets $K\subset
\partial D_{0}\backslash \omega $. Clearly $\mathcal{U}(\omega ,\xi )\subset 
\mathcal{U}_{0}(\omega ,\xi )$, with equality if $\mathbb{C}\backslash
\omega \subset \partial D_{0}$. Nestoridis and Papachristodoulos \cite{NP}
have shown that $\mathcal{U}_{0}(\omega ,\xi )$ is a dense $G_{\delta }$
subset of the space of all holomorphic functions on $\omega $. They further
observed that, if $f\in \mathcal{U}_{0}(\omega ,\xi )$ and $\partial
D_{0}\backslash \omega $ contains a nondegenerate arc, then $f$ does not
extend continuously to $\omega \cup \partial D_{0}$. We can now give:

\begin{corollary}
\label{Cor}Let $f\in \mathcal{U}_{0}(\omega ,\xi )$ and suppose that $\sigma
(\partial D_{0}\backslash \omega )>0$. Then, for $\sigma $-almost every $%
\zeta \in \partial D_{0}\backslash \omega $, the set $f(\Gamma )$ is dense
in $\mathbb{C}$ for every open triangle $\Gamma \subset D_{0}$ which has a
vertex at $\zeta $ and is symmetric about $[0,\zeta ]$.
\end{corollary}

This follows immediately from Theorem \ref{main}, because Plessner's theorem
(Theorem 2.5 of \cite{GM}) tells us that at $\sigma $-almost every point of $%
\partial D_{0}\backslash \omega $ either $f$ has a finite nontangential
limit or $f(\Gamma )$ is dense in $\mathbb{C}$ for every such triangle $%
\Gamma $. The special case of this corollary where $\omega =D_{0}$ was
recently established in \cite{GAIF}. (It was stated there for $f\in \mathcal{%
U}(\omega ,\xi )$, but the proof is valid also for $f\in \mathcal{U}%
_{0}(\omega ,\xi )$.)

Theorem 1 of \cite{GaMa} tells us that, if $\zeta \in \partial
D_{0}\backslash \omega $ and a function $f$ in $\mathcal{U}(\omega ,\xi )$
is bounded in $D(\zeta ,\rho )\cap \omega $ for some $\rho >0$, then $%
\mathbb{C}\backslash (\omega \cup \overline{D}_{0})$ must be polar.
Corollary \ref{Cor} yields the additional information that $(\partial
D_{0}\backslash \omega )\cap D(\zeta ,\rho )$ must have zero arc length
measure.

Our proof of Theorem \ref{main}\ relies on the following subtle convergence
result for harmonic measures, which is of interest in its own right. In what
follows $\Omega $ denotes a domain in $\mathbb{R}^{N}$ ($N\geq 2$)
possessing a Green function $G_{\Omega }(\cdot ,\cdot )$. For any
(non-empty) open set $\omega $, any Borel set $A$ and any point $x$ in $%
\mathbb{R}^{N}$, we denote by $\mu _{x}^{\omega }(A)$ the harmonic measure
of $A\cap \partial \omega $ for $\omega $ evaluated at $x$. (If $x\not\in
\omega $, this measure is assigned the value $0$.)

\begin{theorem}
\label{tool}Let $\xi _{0}\in \Omega $ and $\omega $ be an open subset of $%
\Omega $. Suppose that $(v_{k})$ is a decreasing sequence of subharmonic
functions on $\omega $ such that $v_{1}/G_{\Omega }(\xi _{0},\cdot )$ is
bounded above and $\lim_{k\rightarrow \infty }v_{k}<0$ on $\omega $. If $\mu
_{x_{1}}^{\omega }(\partial \Omega )>0$ for some $x_{1}$, then $\mu
_{x_{1}}^{\{v_{k}<0\}}(\partial \Omega )>0$ for all sufficiently large $k$.
\end{theorem}

The above result fails without the upper boundedness hypothesis on $%
v_{1}/G_{\Omega }(\xi _{0},\cdot )$, as can be seen from the following
examples (there are obvious analogues in higher dimensions):

\begin{enumerate}
\item[(a)] $\Omega =\omega =\mathbb{D}$ and $v_{k}(z)=1+k\log \left\vert
z\right\vert $, so $\{v_{k}<0\}=\{\left\vert z\right\vert <e^{-1/k}\}.$

\item[(b)] $\Omega =\omega =\mathbb{D}$ and $v_{k}(z)=1-k\dfrac{1-\left\vert
z\right\vert ^{2}}{\left\vert 1-z\right\vert ^{2}}$, so $\{v_{k}<0\}$ is a
disc internally tangent to $\mathbb{T}$ at $1.$
\end{enumerate}

A weaker version of this result, where $\Omega $ is a simply connected plane
domain and each function $v_{k}$ is harmonic on all of $\Omega $, was
established in \cite{G13}. We will use a substantially different argument to
prove this more general theorem. When $N=2$ the result is valid for domains
in the extended complex plane $\widehat{\mathbb{C}}=\mathbb{C\cup \{\infty \}%
}$. In the application of Theorem \ref{tool} to the proof of Theorem \ref%
{main} it is crucial that, in contrast to the above two examples, the
sequence $(v_{k})$ need only have a negative limit on a suitable open subset 
$\omega $ of $\Omega $, namely one for which $\mu _{x}^{\omega }(\partial
\Omega )>0$.

Theorem \ref{tool} and its proof are based on Chapter 6 of the second
author's doctoral thesis \cite{MM}.

\section{Proof of Theorem \protect\ref{main}}

Let $f$, $(S_{n_{k}})$, $E$ and $F$ be as in the statement of Theorem \ref%
{main}, and let 
\begin{equation*}
D=\{\zeta \in E\cap F:S(\zeta )\neq f(\zeta )\}.
\end{equation*}%
Also, let $\Gamma (1)$ denote the open triangular region with vertices $%
1,(1\pm i)/2$ (say), and let 
\begin{equation*}
\Gamma (\zeta )=\{\zeta z:z\in \Gamma (1)\}\ \ \ \ (\zeta \in \mathbb{T}).
\end{equation*}

Now suppose, for the sake of contradiction, that the conclusion of the
theorem fails. Then we may choose a positive number $a$ sufficiently large
to ensure that $\sigma (A_{a})>0$, where 
\begin{equation*}
A_{a}=\{\zeta \in D:\left\vert f\right\vert \leq a\text{ on }\Gamma (\zeta )%
\text{ and }\left\vert S_{n_{k}}(\zeta )\right\vert \leq a\text{ for all }%
k\}.
\end{equation*}%
It follows, on multiplication by a suitable unimodular constant, that we can
choose a compact set $K$ of $A_{a}$ such that $\inf_{K}\func{Re}(S-f)>0$ and 
$0<\sigma (K)<2\pi $. The domain $\Omega =\widehat{\mathbb{C}}\backslash K$
then possesses a Green function, by Myrberg's theorem (Theorem 5.3.8 of \cite%
{AG}), since $K$ is non-polar.

We put $\omega =\cup _{\zeta \in K}\Gamma (\zeta )$, and reduce $K$, if
necessary, to ensure that $\omega $ is a simply connected domain. Clearly $%
\left\vert f\right\vert \leq a$ on $\omega $. Since the triangles $\Gamma
(\zeta )$ are congruent, the boundary of $\omega $ is a rectifiable Jordan
curve. Thus $\mu _{z}^{\omega }(K)>0$ when $z\in \omega $ by the F. and M.
Riesz theorem (Theorem VI.1.2 of \cite{GM}), in view of the fact that $%
\sigma (K)>0$. Let $g:\mathbb{D}\rightarrow \omega $ be a conformal map. It
extends to a continuous bijection $g:\overline{\mathbb{D}}\rightarrow 
\overline{\omega }$, by Carath\'{e}odory's theorem. The function $g^{\prime
} $ belongs to the Hardy space $H^{1}$ by Theorem VI.1.1 of \cite{GM}, so
the F. and M. Riesz theorem shows further that, for almost every $\zeta \in
K $, the function $g$ is conformal at $g^{-1}(\zeta )$ and $\mathrm{nt}%
\lim_{w\rightarrow g^{-1}(\zeta )}(f\circ g)(w)=f(\zeta )$. Since $f\circ g$
is a bounded holomorphic function on $\mathbb{D}$, we know that $f\circ
g=H_{f\circ g}^{\mathbb{D}}$, using the usual notation for Dirichlet
solutions, whence 
\begin{equation}
f=H_{f\circ g}^{\mathbb{D}}\circ g^{-1}=H_{f}^{\omega }\text{ \ \ on \ \ }%
\omega .  \label{fHf}
\end{equation}

We define 
\begin{equation*}
u_{k}=\frac{1}{n_{k}}\log \frac{\left\vert S_{n_{k}}-f\right\vert }{2a}\text{
\ on \ }\mathbb{D}\text{ \ \ \ \ }(k\in \mathbb{N}).
\end{equation*}%
Noting from Bernstein's lemma (Theorem 5.5.7 of \cite{Ran}) that%
\begin{equation*}
\left\vert S_{n_{k}}\right\vert \leq ae^{n_{k}G_{\Omega }(\infty ,\cdot )}%
\text{ \ \ on \ }\Omega ,
\end{equation*}%
we see that $u_{k}\leq G_{\Omega }(\infty ,\cdot )$ on $\omega $. Now $\lim
\sup_{k\rightarrow \infty }u_{k}(z)\leq \log \left\vert z\right\vert $ on $%
\mathbb{D}$, so we can choose a sequence $(r_{k})$ in $[0,1)$\ such that $%
r_{k}\uparrow 1$ and 
\begin{equation*}
u_{j}(z)\leq \frac{1}{2}\log \left\vert z\right\vert \text{ \ \ \ }%
(\left\vert z\right\vert \leq r_{k},j\geq k).
\end{equation*}%
Let $v_{k}=H_{\psi _{k}}^{\omega }$, where 
\begin{equation*}
\psi _{k}(z)=\left\{ 
\begin{array}{cc}
\frac{1}{2}\log \left\vert z\right\vert & \text{ \ on \ }\partial \omega
\cap D(0,r_{k}) \\ 
G_{\Omega }(\infty ,z) & \text{ \ on \ }\partial \omega \cap (\mathbb{D}%
\backslash D(0,r_{k})) \\ 
0 & \text{ \ on \ }\partial \omega \cap \mathbb{T}%
\end{array}%
\right. .
\end{equation*}%
Then $u_{k}\leq v_{k}$ on $\omega $ and $(v_{k})$ is a decreasing sequence
of harmonic functions on $\omega $ with limit $\frac{1}{2}\log \left\vert
\cdot \right\vert $ on $\partial \omega $.

By Theorem \ref{tool}, and the fact that $\mu _{z}^{\omega }(K)>0$ when $%
z\in \omega $, there exists $k^{\prime }\in \mathbb{N}$ such that the open
set $\omega _{1}:=\omega \cap \{v_{k^{\prime }}<0\}$ is non-empty and 
\begin{equation}
\mu _{w}^{\omega _{1}}(\partial \Omega )>0\text{ \ for some }w\in \omega
_{1}.  \label{ne}
\end{equation}%
Clearly $u_{k}<0$ on $\omega _{1}$ for all $k\geq k^{\prime }$. Thus $%
\left\vert S_{n_{k}}-f\right\vert \leq 2a$, and so $\left\vert
S_{n_{k}}\right\vert \leq 3a$, on $\omega _{1}$ for all $k\geq k^{\prime }$.
Now $S_{n_{k}}=H_{S_{n_{k}}}^{\omega _{1}}$on $\omega _{1}$, so by dominated
convergence 
\begin{equation*}
f=H_{\phi }^{\omega _{1}}\text{ \ on \ }\omega _{1}\text{, \ \ where \ \ }%
\phi =\left\{ 
\begin{array}{cc}
f & \text{ \ on \ }\partial \omega _{1}\cap \mathbb{D} \\ 
S & \text{ \ on \ }\partial \omega _{1}\cap \mathbb{T}\subset K%
\end{array}%
\right. .
\end{equation*}%
However, we also know from (\ref{fHf}) that 
\begin{equation*}
f=H_{f}^{\omega }=H_{H_{f}^{\omega }}^{\omega _{1}}=H_{f}^{\omega _{1}}\text{
\ on \ }\omega _{1}
\end{equation*}%
(see Theorem 6.3.6\ of \cite{AG}). Thus, by (\ref{ne}) and our choice of $K$%
, we arrive at the contradiction that there is a point $w$ in $\omega _{1}$
satisfying%
\begin{equation*}
0=\func{Re}H_{\phi -f}^{\omega _{1}}(w)\geq \inf_{K}\func{Re}(S-f)\mu
_{w}^{\omega _{1}}(\partial \Omega )>0,
\end{equation*}%
Theorem \ref{main} is now established, subject to verification of Theorem %
\ref{tool}.

\section{Proof of Theorem \protect\ref{tool}}

We will employ some results concerning the Martin boundary and the minimal
fine topology, which are expounded in Chapters 8 and 9 of the book \cite{AG}%
. Let $\widehat{\Omega }=\Omega \cup \Delta $ denote the Martin
compactification of a Greenian domain $\Omega $ in $\mathbb{R}^{N}$, let $%
M(\cdot ,y)$ denote the Martin kernel with pole at $y\in \Delta $, and let $%
\Delta _{1}$ denote the set of minimal elements of $\Delta $.\ Thus 
\begin{equation*}
M(x,y)=\lim_{z\rightarrow y}\frac{G_{\Omega }(x,z)}{G_{\Omega }(x_{0},z)}%
\text{ \ \ \ }(x\in \Omega ,y\in \Delta ),
\end{equation*}%
where $x_{0}$ denotes the reference point for the compactification. A set $%
E\subset \Omega $ is said to be minimally thin at a point $y\in \Delta _{1}$
if $R_{M(\cdot ,y)}^{E}\neq M(\cdot ,y)$, where $R_{u}^{L}$ denotes the
usual reduction of a positive superharmonic function $u$ on $\Omega $
relative to a set $L\subset \Omega $. Further, a function $f$ is said to
have minimal fine limit $l$ at $y$ if there is a set $E$, minimally thin at $%
y$, such that $f(x)\rightarrow l$ as $x\rightarrow y$ in $\Omega \backslash
E $. Limit notions with respect to the minimal fine topology will be
prefixed by \textquotedblleft mf\textquotedblright . The main work lies in
establishing the following result, which develops ideas from \cite{GAIF}.

\begin{proposition}
\label{prop}Let $\xi _{0}\in \Omega ,y\in \Delta _{1}$ and $\omega $ be an
open subset of $\Omega $ such that $\Omega \backslash \omega $ is minimally
thin at $y$. Suppose that $(v_{k})$ is a decreasing sequence of subharmonic
functions on $\omega $ such that $v_{1}/G_{\Omega }(\xi _{0},\cdot )$ is
bounded above and $\lim_{k\rightarrow \infty }v_{k}<0$ on $\omega $. Then
there exists $k^{\prime }\in \mathbb{N}$ such that,%
\begin{equation*}
\mathrm{mf}\lim_{z\rightarrow y}\frac{v_{k}(z)}{G_{\Omega }(\xi _{0},z)}<0%
\text{ \ \ \ }(k\geq k^{\prime }).
\end{equation*}
\end{proposition}

\begin{proof}
Without loss of generality we may assume that $\xi _{0}$ coincides with the
reference point $x_{0}$ for the Martin compactification of $\Omega $, and
that $x_{0}\not\in \overline{\omega }$. For each $k\in \mathbb{N}$ we define 
\begin{equation*}
\widetilde{v}_{k}(z)=\frac{v_{k}(z)}{G_{\Omega }(x_{0},z)}\text{ \ \ \ }%
(z\in \omega ).
\end{equation*}%
By hypothesis there is a positive constant $c$ such that the function $%
cG_{\Omega }(x_{0},\cdot )-v_{1}$ is positive and superharmonic on $\omega $%
. Hence, by Theorem 9.6.2(ii) of \cite{AG}, each function $\widetilde{v}_{k}$
has a minimal fine limit in the range $[-\infty ,c)$ at $y$. We denote this
limit by $\widetilde{v}_{k}(y)$. Thus, for each $k$, there is a set $L_{k}$,
minimally thin at $y$, such that 
\begin{equation*}
\widetilde{v}_{k}(z)\rightarrow \widetilde{v}_{k}(y)\text{ \ \ \ }%
(z\rightarrow y\text{ in }\widehat{\Omega },\text{ }z\in \Omega \backslash
L_{k}).
\end{equation*}%
By Lemma 9.3.1 of \cite{AG} we can now choose a single set $F\subset \Omega $%
, minimally at $y$, such that%
\begin{equation}
\widetilde{v}_{k}(z)\rightarrow \widetilde{v}_{k}(y)\text{ \ \ \ }%
(z\rightarrow y\text{ in }\widehat{\Omega },\text{ }z\in \Omega \backslash F)%
\text{ \ \ for all }k.  \label{F}
\end{equation}

By Corollary 8.2.9 and Theorem 8.3.1 of \cite{AG} we can find an open
neighbourhood $U$ of $\Delta \backslash \{y\}$ in $\widehat{\Omega }$ such
that $U$ is minimally thin at $y$, and hence a closed subneighbourhood $L$
of $\Delta \backslash \{y\}$ with the same property. (A more detailed
explanation of this step may be found in Lemma 7.2.3 of \cite{MM}.) By
removing $L$ from $\omega $ we can ensure that the closure $\overline{\omega 
}^{\widehat{\Omega }}$ of $\omega $ in $\widehat{\Omega }$ meets $\Delta $
precisely at $y$. Next, by Lemma 9.2.2(iii) of \cite{AG}, we can find an
open neighbourhood of $\partial \omega \cap \Omega $ that is minimally thin
at $y$, and hence a subneighbourhood $F_{0}$ of $\partial \omega \cap \Omega 
$ that is closed relative to $\Omega $ and has the same property. We now
define the open set $\omega _{0}=\omega \backslash F_{0}$. Thus $\overline{%
\omega }_{0}\cap \Omega \subset \omega $, the set $\Omega \backslash \omega
_{0}$ is minimally thin at $y$, and $\overline{\omega }_{0}^{\widehat{\Omega 
}}\cap \Delta =\{y\}$. We are going to construct a probability measure $\nu $
on the boundary $\partial ^{\widehat{\Omega }}\omega _{0}$ of $\omega _{0}$
in $\widehat{\Omega }$ satisfying $\nu (\overline{\omega }_{0}^{\widehat{%
\Omega }}\cap \Omega )=1$, whence $\nu (\{y\})=0$, and also 
\begin{equation*}
\widetilde{v}_{k}(y)\leq \int_{\overline{\omega }_{0}^{\widehat{\Omega }}}%
\widetilde{v}_{k}(\zeta )d\nu (\zeta )\text{ \ \ \ }(k\in \mathbb{N}).
\end{equation*}

Let $(\Omega _{m})$ be an exhaustion of $\Omega $ by bounded open sets
satisfying $\overline{\Omega }_{m}\subset \Omega _{m+1}$ for all $m$, and
define $m(z)=\min \{m:z\in \Omega _{m}\}$ for $z\in \Omega $. For each $z\in
\omega _{0}$ we define a measure on $\partial (\Omega _{m(z)}\cap \omega
_{0})$\ by writing 
\begin{equation*}
d\mu _{z}^{\ast }(\zeta )=\frac{G_{\Omega }(x_{0},\zeta )}{G_{\Omega
}(x_{0},z)}d\mu _{z}^{\Omega _{m(z)}\cap \omega _{0}}(\zeta ).
\end{equation*}%
Since $G_{\Omega }(x_{0},\cdot )$ is harmonic on a neighbourhood of $%
\overline{\Omega _{m(z)}\cap \omega _{0}}$,%
\begin{equation*}
\mu _{z}^{\ast }(\overline{\omega }_{0}^{\widehat{\Omega }})=\frac{1}{%
G_{\Omega }(x_{0},z)}\int_{\partial (\Omega _{m(z)}\cap \omega
_{0})}G_{\Omega }(x_{0},\zeta )d\mu _{z}^{\Omega _{m(z)}\cap \omega
_{0}}(\zeta )=1.
\end{equation*}%
Later\ we will arrive at the desired measure $\nu $ as a $w^{\ast }$-limit
of a suitable sequence of measures $(\mu _{z_{n}}^{\ast })$.

As a first step we show that there is a potential $u$ on $\Omega $ and a set 
$E_{0}\subset \Omega $, minimally thin at $y$, such that 
\begin{equation}
\frac{u(z)}{G_{\Omega }(x_{0},z)}\rightarrow \infty \text{ \ \ \ }%
(z\rightarrow y,z\in \Omega \backslash \omega _{0})  \label{C1}
\end{equation}%
and%
\begin{equation}
\frac{u(z)}{G_{\Omega }(x_{0},z)}\rightarrow 1\text{ \ \ \ }(z\rightarrow
y,z\in \Omega \backslash E_{0}).  \label{C2}
\end{equation}%
To see this, we note from Theorem 9.2.7 of \cite{AG} that, since $\Omega
\backslash \omega _{0}$ is minimally thin at $y$, there is a potential $%
G_{\Omega }\mu $ such that 
\begin{equation*}
a:=\int_{\Omega }M(x,y)d\mu (x)<\infty
\end{equation*}%
and%
\begin{equation*}
\frac{G_{\Omega }\mu (z)}{G_{\Omega }(x_{0},z)}\rightarrow \infty \text{ \ \
\ }(z\rightarrow y,z\in \Omega \backslash \omega _{0}).
\end{equation*}%
Also, Fatou's lemma implies that 
\begin{equation*}
\underset{z\rightarrow y}{\lim \inf }\frac{G_{\Omega }\mu (z)}{G_{\Omega
}(x_{0},z)}\geq \int_{\Omega }\underset{z\rightarrow y}{\lim \inf }\frac{%
G_{\Omega }(x,z)}{G_{\Omega }(x_{0},z)}d\mu (x)=\int_{\Omega }M(x,y)d\mu
(x)=a,
\end{equation*}%
while the reverse inequality follows from the result cited above and the
fact that $\Omega $ is not minimally thin at $y$. Hence, by Theorem 9.3.3 of 
\cite{AG}, there is a set $E_{0}\subset \Omega $, minimally thin at $y$,
such that 
\begin{equation*}
\frac{G_{\Omega }\mu (z)}{G_{\Omega }(x_{0},z)}\rightarrow a\text{ \ \ \ }%
(z\rightarrow y,z\in \Omega \backslash E_{0}).
\end{equation*}%
We now obtain (\ref{C1}) and (\ref{C2}) by setting $u=a^{-1}G_{\Omega }\mu $.

Let $\varepsilon >0$. Using the above fact, we can find $r_{\varepsilon }>0$
such that 
\begin{equation*}
u(z)>\frac{G_{\Omega }(x_{0},z)}{\varepsilon }\text{ \ \ \ if \ }z\in
(\Omega \backslash \omega _{0})\cap B_{M}(y,r_{\varepsilon })
\end{equation*}%
and%
\begin{equation*}
u(z)<2G_{\Omega }(x_{0},z)\text{ \ \ \ if \ }z\in (\Omega \backslash
E_{0})\cap B_{M}(y,r_{\varepsilon }),
\end{equation*}%
where $B_{M}(y,r)$ denotes the open ball of centre $y$ and radius $r>0$ with
respect to some metric compatible with the Martin topology. Since $\overline{%
\Omega _{m(z)}\cap \omega _{0}}\subset \Omega $ and $u$ is positive and
superharmonic on $\Omega $, we deduce that, for each $z\in (\omega
_{0}\backslash E_{0})\cap B_{M}(y,r_{\varepsilon })$,%
\begin{eqnarray*}
\mu _{z}^{\ast }(\partial ^{\widehat{\Omega }}\omega _{0}\cap
B_{M}(y,r_{\varepsilon })) &=&\frac{1}{G_{\Omega }(x_{0},z)}\int_{\partial ^{%
\widehat{\Omega }}\omega _{0}\cap B_{M}(y,r_{\varepsilon })}G_{\Omega
}(x_{0},\zeta )~d\mu _{z}^{\Omega _{m(z)}\cap \omega _{0}}(\zeta ) \\
&\leq &\frac{1}{G_{\Omega }(x_{0},z)}\int_{\partial (\Omega _{m(z)}\cap
\omega _{0})}\varepsilon u(\zeta )~d\mu _{z}^{\Omega _{m(z)}\cap \omega
_{0}}(\zeta ) \\
&\leq &\frac{\varepsilon u(z)}{G_{\Omega }(x_{0},z)}\leq 2\varepsilon .
\end{eqnarray*}

Since $E_{0}\cup F$ and $\Omega \backslash \omega _{0}$ are both minimally
thin at $y$, we can choose a sequence $(z_{n})$ in $\omega _{0}\backslash
(E_{0}\cup F)$ such that $z_{n}\rightarrow y$. Thus, recalling (\ref{F}), we
see that%
\begin{equation}
\widetilde{v}_{k}(z_{n})\rightarrow \widetilde{v}_{k}(y)\text{ \ \ \ }%
(n\rightarrow \infty )  \label{seq}
\end{equation}%
and%
\begin{equation}
\mu _{z_{n}}^{\ast }(\partial ^{\widehat{\Omega }}\omega _{0}\cap
B_{M}(y,r_{\varepsilon }))\leq 2\varepsilon \text{ \ for all sufficiently
large }n.  \label{2e}
\end{equation}%
Further, since $(\mu _{z_{n}}^{\ast })$ is a sequence of probability
measures on the compact set $\overline{\omega }_{0}^{\widehat{\Omega }}$,
there is a subsequence $(\mu _{z_{n_{j}}}^{\ast })$ which is $w^{\ast }$%
-convergent to some measure $\nu $. Since every upper bounded upper
semicontinuous function $\phi $ on $\overline{\omega }_{0}^{\widehat{\Omega }%
}$ is the pointwise limit of a decreasing sequence of continuous functions,
the monotone convergence theorem yields%
\begin{equation}
\underset{j\rightarrow \infty }{\lim \sup }\int_{\overline{\omega }_{0}^{%
\widehat{\Omega }}}\phi ~d\mu _{z_{n_{j}}}^{\ast }\leq \int_{\overline{%
\omega }_{0}^{\widehat{\Omega }}}\phi ~d\nu .  \label{usc}
\end{equation}%
Clearly $\nu $ is a probability measure with support in $\partial ^{\widehat{%
\Omega }}\omega _{0}$. Also, for any $\varepsilon >0$, there exists $%
r_{\varepsilon }>0$ such that, by (\ref{2e}),%
\begin{equation*}
\nu (\{y\})\leq \nu (\partial ^{\widehat{\Omega }}\omega _{0}\cap
B_{M}(y,r_{\varepsilon }))\leq 2\varepsilon ,
\end{equation*}%
so $\nu (\{y\})=0$. Since $\overline{\omega }_{0}^{\widehat{\Omega }}\cap
\Delta =\{y\}$, we conclude that $\nu (\partial ^{\widehat{\Omega }}\omega
_{0}\cap \Omega )=1$.

The subharmonicity of $v_{k}$ on $\omega $ implies that%
\begin{eqnarray}
\widetilde{v}_{k}(z_{n_{j}}) &=&\frac{v_{k}(z_{n_{j}})}{G_{\Omega
}(x_{0},z_{n_{j}})}  \notag \\
&\leq &\frac{1}{G_{\Omega }(x_{0},z_{n_{j}})}\int_{\partial (\Omega
_{m(z_{n_{j}})}\cap \omega _{0})}v_{k}(\zeta )~d\mu _{n_{j}}^{\Omega
_{m(z_{n_{j}})}\cap \omega _{0}}(\zeta )  \notag \\
&=&\int_{\overline{\omega }_{0}^{\widehat{\Omega }}}\widetilde{v}_{k}(\zeta
)~d\mu _{z_{n_{j}}}^{\ast }(\zeta ).  \label{smv}
\end{eqnarray}%
Also, the functions $\widetilde{v}_{k}$ are upper semicontinuous on $\omega $
and bounded above (by $c$) on $\overline{\omega }_{0}^{\widehat{\Omega }}$.
Hence, defining $\phi =\widetilde{v}_{k}$ on $\overline{\omega }_{0}^{%
\widehat{\Omega }}\cap \Omega $ and $\phi =c$ at $y$, we see from (\ref{usc}%
) that%
\begin{equation*}
\underset{j\rightarrow \infty }{\lim \sup }\int_{\overline{\omega }_{0}^{%
\widehat{\Omega }}}\widetilde{v}_{k}(\zeta )~d\mu _{z_{n_{j}}}^{\ast }(\zeta
)\leq \int_{\overline{\omega }_{0}^{\widehat{\Omega }}}\widetilde{v}%
_{k}(\zeta )~d\nu (\zeta ).
\end{equation*}%
From (\ref{seq}) and (\ref{smv}) we conclude that 
\begin{equation*}
\widetilde{v}_{k}(y)\leq \int_{\overline{\omega }_{0}^{\widehat{\Omega }}}%
\widetilde{v}_{k}(\zeta )~d\nu (\zeta )\text{ \ \ \ }(k\in \mathbb{N}).
\end{equation*}%
Finally, $(\widetilde{v}_{k})$ is a decreasing sequence of upper bounded
functions on $\overline{\omega }_{0}^{\widehat{\Omega }}$, so we can apply
the monotone convergence theorem to conclude that 
\begin{equation*}
\lim_{k\rightarrow \infty }\widetilde{v}_{k}(y)\leq \int_{\overline{\omega }%
_{0}^{\widehat{\Omega }}}\lim_{k\rightarrow \infty }\widetilde{v}_{k}(\zeta
)d\nu (\zeta ).
\end{equation*}%
Since $\nu (\overline{\omega }_{0}^{\widehat{\Omega }}\cap \Delta )=0$ and $%
\lim_{k\rightarrow \infty }\widetilde{v}_{k}<0$ on $\omega $, we conclude
that $\lim_{k\rightarrow \infty }\widetilde{v}_{k}(y)<0$. Thus $\widetilde{v}%
_{k}(y)<0$ for all sufficiently large $k$, as required.
\end{proof}

\bigskip

\bigskip

\begin{proof}[Proof of Theorem \protect\ref{tool}]
Without loss of generality we may assume that $\omega $ is connected. There
is a (unique) probability measure $\mu _{1}$ on $\Delta _{1}$ such that 
\begin{equation*}
1=\int_{\Delta _{1}}M(x,y)~d\mu _{1}(y)\text{ \ \ \ }(x\in \Omega ).
\end{equation*}%
Hence%
\begin{equation}
\mu _{x}^{\omega }(\Omega )=R_{1}^{\Omega \backslash \omega
}(x)=\int_{\Delta _{1}}R_{M(\cdot ,y)}^{\Omega \backslash \omega }(x)~d\mu
_{1}(y)\text{ \ \ \ }(x\in \omega ),  \label{inter}
\end{equation}%
by Theorem 6.9.1 and Corollary 9.1.4 of \cite{AG}. Thus%
\begin{equation*}
\mu _{x}^{\omega }(\partial \Omega )=1-\mu _{x}^{\omega }(\Omega
)=\int_{A}\left\{ M(x,y)-R_{M(\cdot ,y)}^{\Omega \backslash \omega
}(x)\right\} d\mu _{1}(y)\text{ \ \ \ }(x\in \omega ),
\end{equation*}%
where 
\begin{equation*}
A=\{y\in \Delta _{1}:R_{M(\cdot ,y)}^{\Omega \backslash \omega }\neq M(\cdot
,y)\};
\end{equation*}%
that is, $A$ is the set of points in $\Delta _{1}$ at which $\Omega
\backslash \omega $ is minimally thin. Our hypothesis that $\mu
_{x_{1}}^{\omega }(\partial \Omega )>0$ shows that $\mu _{1}(A)>0$.

Let 
\begin{equation*}
A_{k}=\{y\in A:R_{M(\cdot ,y)}^{\Omega \backslash \{v_{k}<0\}}\neq M(\cdot
,y)\}\text{ \ \ \ }(k\in \mathbb{N}).
\end{equation*}%
Proposition \ref{prop} tells us that, if $y\in A$, then $\Omega \backslash
\{v_{k}<0\}$ is minimally thin at $y$ for all sufficiently large $k$. Hence $%
(A_{k})$ increases to $A$, and so we can choose $k^{\prime }$ such that $\mu
_{1}(A_{k^{\prime }})>0$. On each connected component of the open set $%
\{v_{k^{\prime }}<0\}$ either $R_{M(\cdot ,y)}^{\Omega \backslash
\{v_{k^{\prime }}<0\}}=M(\cdot ,y)$ or $R_{M(\cdot ,y)}^{\Omega \backslash
\{v_{k^{\prime }}<0\}}<M(\cdot ,y)$. Thus we can choose a component $\omega
^{\prime }$ of $\{v_{k^{\prime }}<0\}$ on which $R_{M(\cdot ,y)}^{\Omega
\backslash \{v_{k^{\prime }}<0\}}<M(\cdot ,y)$ for all $y$ in a subset of $%
A_{k^{\prime }}$ of positive $\mu _{1}$-measure. Further, we can arrange
that $x_{1}\in \omega ^{\prime }$ by choosing $k^{\prime }$ large enough.
The preceding calculation, applied to $\{v_{k^{\prime }}<0\}$ and $%
A_{k^{\prime }}$ in place of $\omega $ and $A$, now shows that%
\begin{equation*}
\mu _{x}^{\{v_{k^{\prime }}<0\}}(\partial \Omega )=\int_{A_{k^{\prime
}}}\left\{ M(x,y)-R_{M(\cdot ,y)}^{\Omega \backslash \{v_{k^{\prime
}}<0\}}(x)\right\} d\mu _{1}(y)>0\text{ \ \ \ }(x\in \omega ^{\prime }),
\end{equation*}%
as required.
\end{proof}

\bigskip

\noindent School of Mathematical Sciences,

\noindent University College Dublin,

\noindent Belfield, Dublin 4, Ireland.

\noindent e-mail: stephen.gardiner@ucd.ie

\bigskip

\noindent Department of Mathematics,

\noindent University of Western Ontario,

\noindent London, Ontario, Canada N6A 5B7.

\noindent e-mail: arhimidis8@yahoo.gr

\end{document}